\documentclass[amsthm,amscd,amssymb,verbatim,epsf,amsmath,amsfonts]{amsart}

\usepackage[colorlinks=true,linkcolor=blue,citecolor=blue]{hyperref}

\def\line#1{\hbox to \hsize{#1\hfill}}
\begin{document}
\theoremstyle{plain}
\newtheorem{Thm}{Theorem}
\newtheorem{Cor}{Corollary}
\newtheorem{Con}{Conjecture}
\newtheorem{Main}{Main Theorem}
\newtheorem{Lem}{Lemma}
\newtheorem{Prop}{Proposition}
\def\R{\mathbb{R}}
\def\F{\mathbb{F}}
\def\Re{{\frak R\frak e}}
\def\Im{{\frak I\frak m}}
\def\S{\mathbb{S}}
\def\H{\mathbb{H}}
\def\L{\mathbb{L}}
\theoremstyle{definition}
\newtheorem{Def}{Definition}
\newtheorem{Note}{Note}

\newtheorem{example}{\indent\sc Example}

\theoremstyle{remark}
\newtheorem{notation}{Notation}
\renewcommand{\thenotation}{}

\errorcontextlines=0
\numberwithin{equation}{section}
\renewcommand{\rm}{\normalshape}%

\title[]%
   {Minimal surfaces in the product of two dimensional real space forms endowed with a neutral metric}
\author{ Martha P. Dussan, Nikos Georgiou, Martin Magid}
\address{}
\email{}

\date{12 March 2016}

\address{Martha P. Dussan\\
     Universidade de S\~ao Paulo\\
     Departamento de Matem\'atica -IME\\
    CEP: 05508-090\\
     S\~ao Paulo, Brazil}


\address{Nikos Georgiou\\
 Department of Computing and Mathematics \\
 Waterford Institute of Technology \\
Waterford \\
Ireland.}


\address{Martin Magid\\
 Department of Mathematics \\
Wellesley College \\
Wellesley, MA 02481\\
USA.}

\maketitle

\let\thefootnote\relax\footnote{}
\begin{abstract}We investigate minimal surfaces in products of two-spheres $\S^2_p\times\S^2_p$, with the neutral metric given by $(g,-g)$.  Here $\S^2_p\subset \R^{p,3-p}$ , and $g$ is the induced metric on the sphere.  We compute all totally geodesic surfaces and we give a relation between minimal surfaces and the solutions of the Gordon equations. Finally, in some cases we give a topological classification of compact minimal surfaces.
\end{abstract}
\section*{Introduction}   There has been an explosion of recent papers which consider products of space of constant curvature as ambient spaces.  In particular Harold Rosenberg and his co-authors have looked at constant mean curvature submanifolds in products for a long time, beginning with \cite{R}. The surface theory in products of two-spheres or the product of hyperbolic planes has been of great interest and has been studied extensively in the  two articles \cite{CU} and \cite{urbano2}.

In this article, our ambient spaces are the products $\S^2\times \S^2$, $\H^2\times \H^2$ and $d\S^2\times d\S^2$. Each such four manifold is the space $\L^+(M^3)$ (resp. $\L^-(M^3)$) of spacelike (resp. timelike) oriented geodesics in a certain 3-manifold $(M^3,g)$. In particular, $\L^+(\S^3)=\S^2\times \S^2$, $\L^+(Ad\S^3)=\H^2\times \H^2$ and $\L^-(Ad\S^3)=d\S^2\times d\S^2$ and we simply write $\L^{\pm}(M^3)$ to denote these three manifolds. It is well known, that $\L^{\pm}(M^3)$ enjoys two natural K\"ahler or para-K\"ahler structures with the same symplectic structure $\Omega$ such that the one metric (denoted by $G$) is of neutral signature, locally conformally flat and scalar flat, while the other (denoted by $\bar G$) is Einstein (see \cite{AGK} and \cite{An4}). Furthermore, $G$ and $\bar G$ are invariant under the group action of the isometry group $g$. 

There is interesting relation between the surface theory of the space $\L^{\pm}(M^3)$ of oriented geodesics with the surface theory of $M^3$. In particular, the set of oriented geodesics $\L(S)$ that are orthogonal to a surface $S$ in $M^3$ form a Lagrangian surface in $\L^{\pm}(M^3)$ with respect to the symplectic structure $\Omega$. For example, $L(S)$ is flat with respect to the neutral metric $G$, if and only if, $S$ is Weingarten, that is, its principal curvatures are functionally related (see \cite{An4}, \cite{GG} and \cite{GK}). Furthermore, $G$-minimal Lagrangian surfaces in $\L^{\pm}(M^3)$ are orthogonal to equidistant tubes along a geodesic in $M^3$. The study of minimal Lagrangian surfaces in $\L^{\pm}(M^3)$ with respect to the Einstein metric $G'$ have been studied by many authors (see \cite{CU}, \cite{palmer} and \cite{urbano2}).

Torralbo and Urbano in \cite{torurb}, have studied $G'$-minimal surfaces in $\L(\S^3)=\S^2 \times \S^2$ which, inspired us to investigate $G$-minimal surfaces in $\L^{\pm}(M^3)=\S^2\times \S^2, \H^2\times \H^2$ and $d\S^2\times d\S^2$. This leads to new, interesting examples of minimal surfaces in $\S^2\times \S^2$ and $d\S^2\times d\S^2$, as well as a classification of the totally geodesic surfaces. Our debt to them is clear throughout this paper.

\vspace{0.1in}

\section{Surfaces in ${\mathbb S}_{p,1}^2\times {\mathbb S}_{p,1}^2$ endowed with the neutral metric}\label{s:construction}

Consider the real space ${\mathbb R}^3$, endowed with the pseudo-Riemannian metric 
\[
\left<.,.\right>_p=-\sum_{i=1}^p dx_i^2+\sum_{i=p+1}^3 dx_i^2,
\]
of signature $(p,3-p)$, with $p=1,2$. For $p=0$, define $\left<.,.\right>_0=dx_1^2+dx_2^2+dx_3^2$.

For $\delta\in \{-1,1\}$, we define the 2-dimensional complete quadric ${\mathbb S}_{p,\delta}^2:=\{x\in {\mathbb R}^3|\; \left<x,x\right>_p=\delta\}$, endowed with the induced metric $g$ obtained by the canonical inclusion ${\mathbb S}_{p,\delta}^2\hookrightarrow ({\mathbb R}^3,\left<.,.\right>_p)$. The metric $g$ has signature $(p,2-p)$ if $\delta=1$ and  $(p-1,3-p)$ if $\delta=-1$. Furthermore $g$ is of constant Gauss curvature $K^g=\delta$.

Note that the 2-sphere ${\mathbb S}^2={\mathbb S}_{0,1}^2$ and the hyperbolic plane ${\mathbb H}^2$, which is anti-isometric to ${\mathbb S}_{2,1}^2\cap \{x\in {\mathbb R}^3|\; x_3>0\}$, are the only Riemannian quadrics, while the \emph{de Sitter 2-space}  ${\mbox d}{\mathbb S}^2={\mathbb S}_{1,1}^2$, which is anti-isometric to the \emph{anti de Sitter 2-space} ${\mbox Ad}{\mathbb S}^2={\mathbb S}_{2,-1}^2$, are the only Lorentzian quadrics. Since we can get all the cases by looking at $-g$  we only consider the case where $\delta=1$.  For simplicity we replace ${\mathbb S}_{p,1}^2$ with ${\mathbb S}_{p}^2$ 

If $g$ is Riemannian, we can define a complex structure $j$  in $T{\mathbb S}_p^2$  as $j_x(v):=-x \times v$.  The formula
$$\langle a \times b, a \times d \rangle = \langle a, a \rangle \langle b, d \rangle- \langle a, d \rangle \langle b, a \rangle$$
 shows that $j$ is an isometry.

 In the case where ${\mathbb S}_p^2$ is the de Sitter 2-space ${\mbox d}{\mathbb S}^2$, the paracomplex structure $j$ can be defined as follows: If ${\mathbb R}^{1,2}$ denotes the Lorentzian space $({\mathbb R}^3,\left<,\right>_1)$, we define the Lorentzian cross product $\otimes$ in ${\mathbb R}^{1,2}$ by
\[
 u\otimes v:={\mbox I}_{1,2}\cdot (u\times v), 
\]
where $u\times v$ is the standard cross product in ${\mathbb R}^3$ and ${\mbox I}_{1,2}={\mbox {diag}}(-1,1,1)$. For $u,v,w\in {\mathbb R}^3$ we have
\begin{equation}\label{e:lorcrosspro}
\left<u\otimes v,u\otimes w\right>_1=-\left<u,u\right>_1\left<v,w\right>_1+\left<u,v\right>_1\left<u,w\right>_1.
\end{equation}
The paracomplex structure $j$ on ${\mbox d}{\mathbb S}^2$ is given by $j_x(v):=-x\otimes v$, where $v\in {\mathbb R}^3$ is such that $\left<x,v\right>_1=0$. It can be verified easily that $x\otimes (x\otimes v)=v$.

For the triples $({\mathbb S}_p^2,g,j)$, where $j$ is the (para-)complex structure as defined before, (i.e., $j^2=(-1)^{p+1}\mbox{Id}$) we may define the K\"ahler 2-form, $\omega(.,.):=g(j.,.)$ and therefore the quadruples $({\mathbb S}_p^2,g,j,\omega)$ are 2-dimensional K\"ahler structures for $p=0,2$, while for $p=1$ the quadraple $(d{\mathbb S}^2,g,j,\omega)$ is a para-K\"ahler structure.

\vspace{0.1in}

On the product ${\mathbb S}_p^2\times {\mathbb S}_p^2$, define two (para-)complex structures $J_1,J_2$ by
\[
J_1=j\oplus j,\qquad J_2=j\oplus -j,
\]
and the symplectic structures $\Omega_1,\Omega_2$ by
\[
\Omega_1=\pi^{\ast}_1\omega-\pi^{\ast}_2\omega,\qquad \Omega_2=\pi^{\ast}_1\omega+\pi^{\ast}_2\omega.
\]
We now define the metric $G$ in ${\mathbb S}_p^2\times {\mathbb S}_p^2$ by
\[
G(.,.)=(-1)^{p+1}\Omega_1(J_1.,.)=(-1)^{p+1}\Omega_2(J_2.,.).
\]
Observe that $G$ is of signature $(++--)$ and is called a \emph{neutral metric}. The Levi-Civita connection $\nabla$ with respect to the metric $G$ is given by
\[
\nabla_X Y=(D_{X_1} Y_1,D_{X_2} Y_2),
\]
where $X=(X_1,X_2),Y=(Y_1,Y_2)$ are vector fields in ${\mathbb S}_p^2\times {\mathbb S}_p^2$ and $D$ denote the Levi-Civita connection with respect to $g$. If $p=0$ or $p=2$, the endomorphisms $J_1,J_2$ are both almost complex structures and then 
\[
\Omega_k(X,Y)=\Omega_k(J_kX,J_kY)=-G(X,J_kY)=G(J_kX,Y).
\]
If $p=1$, the endomorphisms $J_1,J_2$ are both almost paracomplex structures. Thus,
\[
\Omega_k(X,Y)=-\Omega_k(J_kX,J_kY)=-G(X,J_kY)=G(J_kX,Y).
\]
We then have:
\begin{Prop}\cite{An4}, \cite{Ge1}
The quadruples $({\mathbb S}_p^2\times {\mathbb S}_p^2, G,J_1,\Omega_1)$ and $({\mathbb S}_p^2\times {\mathbb S}_p^2, G,J_2,\Omega_2)$ are (para-) K\"ahler structures when $({\mathbb S}_p^2,g)$ is Riemannian (Lorentzian). Furthermore, the neutral (para-) K\"ahler metric $G$ is locally conformally flat. 
\end{Prop}

Let $F:\Sigma\rightarrow {\mathbb S}_p^2\times {\mathbb S}_p^2$ be an immersion of an oriented surface $\Sigma$. Associated to the two (para-) K\"ahler structures $(G,J_1,\Omega_1)$ and $(G,J_2,\Omega_2)$ on ${\mathbb S}_p^2\times {\mathbb S}_p^2$, there exist two functions $C_1,C_2$, called \emph{K\"ahler functions}, on $\Sigma$ defined by
\begin{equation}\label{e:defiofck}
F^{\ast}\Omega_k=C_k\omega_{\Sigma},\quad k=1,2,
\end{equation}
where $\omega_{\Sigma}$ denotes the area form of $\Sigma$ with respect to the induced metric $F^{\ast}G$. 
\vskip .1in

\begin{Def}\label{d:lagrangian}
A point $q$ in the surface $\Sigma$ in $({\mathbb S}_p^2\times {\mathbb S}_p^2,G)$ is said to be $\Omega_k${\it -Lagrangian} if $C_k(q)=0$. A {\it Lagrangian surface} is a surface that all points are Lagrangian.
\end{Def}
If $\pi_k:{\mathbb S}_p^2\times {\mathbb S}_p^2\rightarrow {\mathbb S}_p^2: (x_1,x_2)\mapsto x_k$ is the $k$-th projection, and $F$ is an immersion of an oriented surface $\Sigma$ in ${\mathbb S}_p^2\times {\mathbb S}_p^2$, we define the maps $F_1,F_2:\Sigma\rightarrow {\mathbb S}_p^2$ by $F_k:=\pi_k\circ F$, where $k=1,2$ and we write $F=(F_1,F_2)$. The Jacobians of $F_1$ and $F_2$ are defined by
\[
F^{\ast}_k\omega=\mbox{Jac}(F_k)\;\omega_{\Sigma},\quad k=1,2.
\]
Then, $$\mbox{Jac}(F_1)=(C_1+C_2)/2,\qquad \mbox{Jac}(F_2)=(-C_1+C_2)/2.$$ 
\vskip .1in
Note also that we have a class of Lagrangian surfaces:
\begin{example} Any product of non-null curves defined by $F_1(x,y)=\phi(x)$ and $F_2(x,y)=\psi(y)$ which use arc-length parameters  is Lagrangian.\end{example}
Let $\{e_1,e_2\}$ be an orthonormal frame with respect to $F^{\ast}G$ such that $|e_1|^2=\epsilon |e_2|^2=1$, where $\epsilon\in\{-1,1\}$ and is oriented such that $\jmath e_1=e_2$ and $\jmath e_2=-\epsilon e_1$. Then $\jmath$ is a complex structure when $\epsilon=1$ while, $\jmath$ is a paracomplex structure when $\epsilon=-1$. Furthermore, $\jmath$ is compatible with $(F^{\ast}G,\omega_{\Sigma})$ and thus $\omega_{\Sigma}(.,.)=F^{\ast}G(\jmath .,.)$. Hence,
\begin{equation}\label{e:omega1}
\omega_{\Sigma}(e_1,e_2)=F^{\ast}G(\jmath e_1,e_2)=F^{\ast}G(e_2,e_2)=\epsilon.
\end{equation}
Consider an oriented orthonormal frame $(s_1,s_2)$ of ${\mathbb S}^2_p$ such that $|s_1|^2=(-1)^p|s_2|^2=1$ and $js_1=s_2$ and $js_2=(-1)^{p+1} s_1$. Define the functions $\lambda_k,\tilde\lambda_k,\mu_k,\tilde\mu_k$ on $\Sigma$ by
\[
dF_1(e_1)=\lambda_1 s_1+\lambda_2 s_2,\quad dF_1(e_2)=\mu_1 s_1+\mu_2 s_2,
\]
\[
dF_2(e_1)=\tilde\lambda_1 s_1+\tilde\lambda_2 s_2,\quad dF_2(e_2)=\tilde\mu_1 s_1+\tilde\mu_2 s_2
\]
If $R$ and $K^g=1$ denote the Riemann and the Gauss curvature of $g$, respectively, we have
\[
g(R(s_1,s_2)s_2,s_1)=K^g(g(s_1,s_1)g(s_2,s_2)-g(s_1,s_2)^2)=(-1)^{p}.
\]
The Riemann curvature tensor $\bar R$ of $G$ satisfies
\begin{equation}\label{e:riemcurambient}
\bar R(e_1,e_2,e_2,e_1)=(-1)^{p}((\lambda_1\mu_2-\lambda_2\mu_1)^2-
(\tilde\lambda_1\tilde\mu_2-\tilde\lambda_2\tilde\mu_1)^2)
\end{equation}
On the other hand, from (\ref{e:defiofck}) and (\ref{e:omega1}), we have
\begin{eqnarray}
\epsilon C_1&=&\omega_{\Sigma}(e_1,e_2) C_1\nonumber \\
&=&(-1)^{p}(\lambda_1\mu_2-\lambda_2\mu_1-\tilde\lambda_1\tilde\mu_2+\tilde\lambda_2\tilde\mu_1).\label{e:relationforc1}
\end{eqnarray}
Similarly, we find
\begin{equation}\label{e:relationforc2}
\epsilon C_2=(-1)^{p}(\lambda_1\mu_2-\lambda_2\mu_1+\tilde\lambda_1\tilde\mu_2-\tilde\lambda_2\tilde\mu_1),
\end{equation}
and therefore, the relation (\ref{e:riemcurambient}) becomes
\[
\bar R(e_1,e_2,e_2,e_1)=(-1)^{p}C_1C_2.
\]
Note also that, 
\[
R(e_1,e_2,e_2,e_1)= \epsilon K.
\]
If $K$ denotes the Gauss curvature of $F^{\ast}G$, the Gauss equation of $F$ gives:
\begin{equation}\label{e:forthegaussequat}
K=\epsilon (-1)^{p} C_1C_2+2|H|^2-\frac{|h|^2}{2},
\end{equation}
where $h,H$ are the second fundamental form and the mean curvature, respectively and $|h|^2:=|h(e_1,e_1)|^2+|h(e_2,e_2)|^2+2\epsilon |h(e_1,e_2)|^2$. Notice that, if $\{e'_1,e'_2\}$ is another orthonormal frame the function $|h|^2$ is invariant. 

\noindent The following Proposition follows directly from (\ref{e:forthegaussequat}). 

\begin{Prop}
Every totally geodesic Lagrangian immersion in ${\mathbb S}_p^2\times {\mathbb S}_p^2$, with respect to either $(G,\Omega_1,J_1)$ or $(G,\Omega_1,J_1)$, is flat.
\end{Prop}

Using (\ref{e:relationforc1}) and (\ref{e:relationforc2}) we have,
\[
\lambda_1\mu_2-\lambda_2\mu_1=\frac{\epsilon (-1)^{p}(C_1+C_2)}{2}=\epsilon (-1)^{p}\mbox{Jac}(F_1),
\]
\[
\tilde\lambda_1\tilde\mu_2-\tilde\lambda_2\tilde\mu_1=\frac{\epsilon (-1)^{p}(-C_1+C_2)}{2}=\epsilon (-1)^{p}\mbox{Jac}(F_2).
\]
Let $\{e_1,e_2\}$ be an oriented orthonormal frame of $F^{\ast}G$ given as before and let $\{v_1,v_2\}$ be sections of the normal bundle such that $\{e_1,e_2,v_1,v_2\}$ is an oriented orthonormal frame of $G$. Then,
\[
\bar R(e_1,e_2,v_2,v_1)=0.
\]
The curvature $K^{\bot}$ of the normal bundle is $K^{\bot}=R^{\bot}(e_1,e_2,v_2,v_1)$, where $R^{\bot}$ is the Riemann curvature tensor of the normal connection $\nabla^{\bot}$. The Ricci equation gives
\[
K^{\bot}=\bar R(e_1,e_2,v_2,v_1)+G([A_{v_2},A_{v_1}]e_1,e_2),
\]
where $A$ denotes the shape operator of $F$. But, $\bar R(e_1,e_2,v_2,v_1)=0$ and therefore, the Ricci equation becomes
\begin{equation}\label{e:norcarv}
K^{\bot}=G([A_{v_2},A_{v_1}]e_1,e_2).
\end{equation}

\vspace{0.1in}

\section{Minimal surfaces}\label{s:minimalsurfaces}



We now study minimal surfaces in $({\mathbb S}_p^2\times {\mathbb S}_p^2,G)$ which will be referred as \emph{$G$-minimal surfaces}.

\begin{Def}\label{d:paracomplexde}
Let $F$ be an immersion of a surface $\Sigma$ in ${\mathbb S}_p^2\times {\mathbb S}_p^2$. A point in $\Sigma$ is said to be a ({\it para-}) {\it complex point} with respect to the (para-)complex structure $J_k$ if the
tangent plane of the surface is preserved by $J_k$. The immersion $F$ is said to be a ({\it para-}) {\it complex curve} in ${\mathbb S}^2_p\times {\mathbb S}^2_p$ if all points are (para-)complex points with respect to either $J_1$ or $J_2$.
\end{Def}

For (para-)complex curves  in $({\mathbb S}^2_p\times {\mathbb S}^2_p,J_k,\Omega_k,G)$, we have the following:

\begin{Prop}
Any (para-)complex curve in $({\mathbb S}^2_p\times {\mathbb S}^2_p,J_k,\Omega_k,G)$ is a minimal immersion.
\end{Prop}

\vspace{0.1in}

Let $F=(F_1,F_2):\Sigma\rightarrow {\mathbb S}^2_p\times {\mathbb S}^2_p$ be a $G$-minimal immersion of an oriented surface $\Sigma$ in the product ${\mathbb S}^2_p\times {\mathbb S}^2_p$. We consider local isothermic coordinates $(x,y)$ of the induced metric $F^{\ast}G$ such that $\left<F_x,F_x\right>_p=\epsilon\left<F_y,F_y\right>_p=e^{2u}$ and $\left<F_x,F_y\right>_p=0$. 

Note that for $\epsilon=1$, the induced metric is Riemannian, in the sense that every non-zero tangent vector field has positive length. Analogously, we could consider the negative definite case where each tangent vector field has negative length.
\vskip .1in 	
Furthermore we can see that for isothermal coordinates:

\begin{eqnarray*}
C_1\epsilon e^{2u}=g(jF_1(\partial x),F_1(\partial y))-g(jF_2(\partial x),F_2(\partial y))\\
C_2\epsilon e^{2u}=g(jF_1(\partial x),F_1(\partial y))+g(jF_2(\partial x),F_2(\partial y)).
\end{eqnarray*}

Let $\{N,\tilde N\}$ be an orthonormal frame of the normal bundle such that $\{F_x,F_y,N,\tilde N\}$ is an oriented frame in $F^{\ast}T({\mathbb S}_p^2\times {\mathbb S}_p^2)$, with
\begin{equation}\label{e:normalnorm}
|N|^2=\epsilon |\tilde N|^2=-\epsilon b,
\end{equation}
where $b\in\{-1,1\}$. Note that, if $\epsilon=1$, then $b=1$.

We now introduce the conformal variable $z:=x+iy$, with $i^2=-\epsilon$. In other words, $z$ is a complex variable when $\epsilon=1$ while, $z$ is a paracomplex variable when $\epsilon=-1$. We then have the following operators:
\[
\frac{\partial}{\partial z}=\frac{1}{2}\Big(\frac{\partial}{\partial x}-\epsilon i \frac{\partial}{\partial y}\Big),\qquad
\frac{\partial}{\partial \bar z}=\frac{1}{2}\Big(\frac{\partial}{\partial x}+\epsilon i \frac{\partial}{\partial y}\Big),
\]
and 
\[
\left<F_z,F_z\right>_p=\left<(F_1)_z,(F_1)_z\right>_p-\left<(F_2)_z,(F_2)_z\right>_p=0,
\]
\[
|F_z|^2=\left<F_z,F_{\bar z}\right>_p=\left<(F_1)_z,(F_1)_{\bar z}\right>_p-\left<(F_2)_z,(F_2)_{\bar z}\right>_p=\frac{e^{2u}}{2}.
\]
If $\xi=(N-i\epsilon\tilde N)/\sqrt{2}$, we have that $|\xi|^2:=\left<\xi,\bar\xi\right>_p=-\epsilon b$ and $\{\xi,\bar\xi\}$ is an orthonormal frame of the complexified normal bundle. Furthermore,
\begin{equation}\label{e:exprgamma1}
J_1F_z=iC_1F_z+\epsilon\gamma_1\xi,\qquad J_2F_z=iC_2F_z+\epsilon\gamma_2\bar\xi
\end{equation}
where $\gamma_k$ is a (para-)complex function. Note that,
\[
|\gamma_k|^2=\frac{\epsilon b e^{2u}}{2}(\epsilon C_k^2+(-1)^{p+1}) 
\] and $\left<\,,\,\right>_p$ is extended bilinearly to the complex vectors.

With the definitions \ref{d:lagrangian} and \ref{d:paracomplexde} we have the following :
\begin{Prop}
An immersion $F$ is $\Omega_k$-Lagrangian in an open set $U$ iff $C_k$ vanishes on $U$. Moreover, $F$ is $J_k$-(para) complex curve in $U$ iff $C_k^2=1$ on $U$.
\end{Prop}

Away from (para-)complex points (that is, $\epsilon C_k^2+(-1)^{p+1}\neq 0$, or $\gamma_1\gamma_2\neq 0$) let $F=(F_1,F_2)$ be a minimal immersion of a surface in ${\mathbb S}^2_p\times {\mathbb S}^2_p$. If $\hat{F}:=(F_1,-F_2)$, then $\{(F+\hat{F})/2,(F-\hat{F})/2\}$ is an orthogonal frame along $F$ of the normal bundle of ${\mathbb S}^2_p\times {\mathbb S}^2_p$ in ${\mathbb R}^{3-p,p}\times {\mathbb R}^{3-p,p}$. 

From $\hat{F}_z=(-1)^{p+1}J_1J_2 F_z$ and using the fact that $F$ is minimal we obtain:
\begin{equation}\label{e:tothema}
\hat{F}_z=\epsilon (-1)^p(C_1C_2 F_z -2e^{-2u}b\gamma_1\gamma_2F_{\bar z}-i\gamma_1C_2\xi-i\gamma_2 C_1\bar\xi).
\end{equation}
This allows us to compute the fundamental equations of the immersion:
\begin{Prop}\label{p:frenequations}
The Frenet equations of the minimal immersion $F$, away from (para-)complex points, are given by
\begin{eqnarray}
F_{zz}&=&2u_zF_z+f_1\xi+f_2\bar\xi+\epsilon (-1)^{p}\frac{b\gamma_1\gamma_2}{2}F,\label{e:protothema} \\
F_{z\bar z}&=&(-1)^{p+1}\frac{\epsilon C_1C_2 e^{2u}}{4}F-\frac{e^{2u}}{4}\hat F,\label{e:deuterothema} \\
\xi_z&=&2\epsilon e^{-2u}bf_2F_{\bar z}+A\xi
+(-1)^{p+1}\frac{ibC_1\gamma_2}{2}F,\label{e:ena} \\
\bar\xi_{z}&=&2\epsilon e^{-2u}bf_1F_{\bar z}-A\bar\xi
+(-1)^{p+1}\frac{ibC_2\gamma_1}{2}F,\label{e:dio1}
\end{eqnarray}
for certain local (para-)complex functions $f_1,f_2$ and,
\begin{equation}\label{e:themissing}
A=(-1)^{j+1}\left(2 u_{z}-\frac{2i\epsilon C_{j} f_{j}+ (\gamma_{j})_{ z}}{\gamma_{j}}\right).
\end{equation}
\end{Prop}
\begin{proof}
From (\ref{e:tothema}) and $\langle F,\hat{F}\rangle=2$, we find (\ref{e:protothema}) and  (\ref{e:deuterothema}). Using now the facts that $J_1^2=(-1)^{p+1}{\mbox Id}$ and $i^2=-\epsilon$ and (\ref{e:exprgamma1}), we have:
\[
J_1\xi=2b e^{-2u}\bar\gamma_1 F_z-iC_1\xi,\qquad J_2\xi=2b e^{-2u}\gamma_2 F_{\bar z}+iC_2\xi.
\]
Furthermore, (\ref{e:ena}) and (\ref{e:dio1}) come directly from differentiating both expressions of (\ref{e:exprgamma1}) with respect to $z$  and using the following expressions
\[
(J_1F_{z})_z=2(iu_zC_1+be^{-2u}\bar\gamma_1f_1)F_z+2be^{-2u}\gamma_1f_2F_{\bar z}+(2\epsilon u_z\gamma_1
-iC_1f_1)\xi+iC_1f_2\bar\xi,
\]
and 
\[
\nabla^\perp_{\bar z}(J_2F_{z}):=(J_2F_{z})_{\bar z}=\epsilon (-1)^{p+1}\left(\frac{iC_1C_2^2e^{2u}}{4}-\frac{ib|\gamma_2|^2C_1}{2}\right)F-\frac{iC_2e^{2u}}{4}\hat F.
\]
Note that, from $|\xi|^2=-\epsilon b$ together with (\ref{e:ena}) and (\ref{e:dio1}), we have
\[
A=2 u_z-\frac{2i\epsilon C_1f_1+ (\gamma_1)_z}{\gamma_1}=-2 u_z+\frac{2i\epsilon C_2f_2+ (\gamma_2)_z}{\gamma_2}.
\]
\end{proof}

The septuple $(u,C_j,\gamma_j,f_j\; :j=1,2)$ is called the \emph{fundamental data} of the pair $(F,\xi)$. As in the case of \cite{torurb}, if $\{\xi^{\ast},{\bar\xi}^{\ast}\}$ is another orthonormal oriented frame in the complexifed normal bundle then there is a function $\theta$ such that $\xi^{\ast}= {\texttt{exp}\epsilon}(i \theta)\xi$, where 
\[
 {\texttt{exp}\epsilon}(i \theta):= \left\{ 
\begin{array}{ccc} \cos (\theta) + i \sin (\theta) &\mbox{ if }& \epsilon=1, \\
\sinh (\theta) + i \cosh (\theta) &\mbox{ if } &\epsilon=-1.\end{array} \right.
\]
In this case, the fundamental data $(u,C_j,\gamma_j^{\ast},f_j^{\ast})$ of the pair $(F,\xi^{\ast})$ are related to the fundamental data of $(F,\xi)$ as follows
\[
\gamma_1^{\ast}={\texttt{exp}\epsilon}(-i\theta)\gamma_1,\qquad \gamma_2^{\ast}={\texttt{exp}\epsilon}(i \theta)\gamma_2,\qquad f_1^{\ast}={\texttt{exp}\epsilon}(-i \theta)f_1,\qquad f_2^{\ast}={\texttt{exp}\epsilon}(i \theta)f_2.
\]
\begin{Prop}\label{p:funddata}
Let $F:\Sigma\rightarrow {\mathbb S}^2_p\times {\mathbb S}^2_p$ be a minimal immersion of an orientable surface $\Sigma$ and $(u,C_j,\gamma_j,f_j:j=1,2)$ its fundamental data for a given orthonormal frame. For $j=1,2$, away from (para-)complex points, we have:
\begin{eqnarray}
(C_j)_z&=&-2i\epsilon be^{-2u}\bar\gamma_j f_j,\label{e:kahler} \\
(\bar f_j)_{z}&=&(-1)^{j+1}\bar f_j A+\frac{i\epsilon (-1)^{p+1} e^{2u}\bar\gamma_jC_{j'}}{4},\label{e:deroff1}\\
(\bar\gamma_j)_z&=&(-1)^{j+1}\bar\gamma_jA,\label{e:derivofgamma}\\
|\gamma_j|^2&=&\frac{\epsilon b e^{2u}}{2}(\epsilon C^2_j+(-1)^{p+1}).\label{e:gammanorsec}
\end{eqnarray}
\end{Prop}
\begin{proof}
We differentiate with respect to $z$ and $\bar z$ the relations in (\ref{e:exprgamma1}) and then by  (\ref{e:ena}) and (\ref{e:dio1}) we obtain (\ref{e:kahler}) and (\ref{e:derivofgamma}).
Finally using the relation $F_{z\bar z z}=F_{zz\bar z}$ we obtain (\ref{e:deroff1}). 
\end{proof}

Given (para-)complex functions $u,C_j,\gamma_j,f_j$ satisfying (\ref{e:kahler}), (\ref{e:deroff1}) and (\ref{e:derivofgamma}), the following Proposition proves that there exists a unique minimal surface with Riemannian induced metric, whose fundamental data is $(u,C_j,\gamma_j,f_j:j=1,2)$.

\begin{Prop}\label{p:compatibility}
Let $\Sigma$ be a smooth surface and $u,A,C_j,\gamma_j,f_j,j=1,2$, be complex functions with $C_j$ non-constant so that (\ref{e:kahler}), (\ref{e:deroff1}), (\ref{e:derivofgamma}) and (\ref{e:gammanorsec}) hold. Suppose that $\gamma_j$ vanishes in at most isolated points. Then there exist, up to congruences, a unique minimal immersion $F:\Sigma\rightarrow {\mathbb S}^2_p\times {\mathbb S}^2_p$ that is non-(para) complex, and an orthonormal frame of the complexified normal bundle $\{\xi,\bar\xi\}$ whose fundamental data is $(u,C_j,\gamma_j,f_j:j=1,2)$.
\end{Prop}
\begin{proof}
From the definition of $A$, where $\gamma_j\ne 0$ we have
\[
(\gamma_j)_z=(-1)^{j}A\gamma_j+2u_z\gamma_j-2i\epsilon C_jf_j.
\]
From ${\gamma_j}_{z\bar z}={\gamma_j}_{\bar z z}$ and using (\ref{e:kahler}), (\ref{e:deroff1}), (\ref{e:derivofgamma}) and (\ref{e:gammanorsec}), we get
\begin{equation}\label{e:theimportter}
2u_{z\bar z}+4\epsilon e^{-2u}|f_j|^2+(-1)^j(\bar A_z+A_{\bar z})+\epsilon(-1)^p e^{2u}C_1C_2/2=0.
\end{equation}
Using now (\ref{e:theimportter}), a brief computation shows $F_{zz\bar z}=F_{z\bar z z}$ and $\xi_{\bar z z}=\xi_{z\bar z}$, which are the integrability conditions of the Frenet system.
\end{proof}



\subsection{Lagrangian surfaces}
The following Theorem gives a classification of all Lagrangian $G$-minimal surfaces in ${\mathbb S}^2_p\times {\mathbb S}^2_p$.
\begin{Thm}\label{t:claslagrangiansurfaces}
Let $F:\Sigma\rightarrow {\mathbb S}^2_p\times {\mathbb S}^2_p$ be a $G$-minimal immersion. Then $F$ is $\Omega_1$-Lagrangian immersion if and only if $F$ is $\Omega_2$-Lagrangian immersion.
\end{Thm}
\begin{proof}
Suppose that $F$ is $\Omega_m$-Lagrangian $G$-minimal immersion. Then $C_j=0$ and thus $\gamma_j\neq 0$, since otherwise the induced metric is degenerate. From (\ref{e:kahler}), we have that $f_j=0$. Using now (\ref{e:deroff1}) we conclude that  $F$ is also $\Omega_{m'}$-Lagrangian, since $C_{m'}=0$. 
\end{proof}
A generalization of Theorem \ref{t:claslagrangiansurfaces} can be found in \cite{Ge1} and \cite{Ge2}. In particular,  it was proved in these articles that every Lagrangian $G$-minimal surface in ${\mathbb S}^2_p\times {\mathbb S}^2_p$ is locally the product of geodesics in ${\mathbb S}^2_p$. 

\vspace{0.1in}

\subsection{Complex curves}

Suppose that the immersion $F=(F_1,F_2):\Sigma\rightarrow {\mathbb S}^2_p\times {\mathbb S}^2_p$ is a complex curve with respect to the (para-)complex structure $J_k$. Without loss of generality, we assume that $F$ is a complex curve with respect to $J_1$. Then, following the same computation as before (by only considering the local functions $\lambda_2$ and $\mu_2$) and away from (para-)complex points with respect to $J_2$ we have $$J_1 F_z=i C_1 F_z,\quad\mbox{and}\quad J_2 F_z=iC_2 F_z+\epsilon \gamma_2\bar\xi,$$
where $C_1$ is a constant function satisfying $C_1^2=\epsilon (-1)^p=1$. Then the Frenet equations of $F$ are given by the following proposition.
\begin{Prop}\label{p:frenequations1}
The Frenet equations of the $J_1$-complex immersion of $F$ are given by
\begin{eqnarray}
F_{zz}&=&2u_zF_z+f_2\bar\xi, \nonumber\\
F_{z\bar z}&=&(-1)^{p+1}\frac{\epsilon C_1C_2 e^{2u}}{4}F-\frac{e^{2u}}{4}\hat F,\nonumber\\
\xi_z&=&2\epsilon e^{-2u}bf_2F_{\bar z}-\left(2 u_z-\frac{2i\epsilon C_2f_2+ (\gamma_2)_z}{\gamma_2}\right)\xi
+(-1)^{p+1}\frac{ibC_1\gamma_2}{2}F, \nonumber\\
\bar\xi_{z}&=&\left(2 u_z-\frac{2i\epsilon C_2f_2+ (\gamma_2)_z}{\gamma_2}\right)\bar\xi,\nonumber
\end{eqnarray}
for certain local (para-)complex function $f_2$.
\end{Prop}

Additonally, the fundamental data $(u, C_j,\gamma_j,f_j)$ satisfy:

\begin{Prop}\label{p:funddata1}
Let $F:\Sigma\rightarrow {\mathbb S}^2_p\times {\mathbb S}^2_p$ be a $J_1$-complex immersion of an orientable surface $\Sigma$ and $(u,C_2,\gamma_2,f_2:j=1,2)$ its fundamental data for a given orthonormal frame. Then
\[
i(C_2)_z=2be^{-2u}\bar\gamma_2 f_2,\qquad (\bar\gamma_2)_z=\bar\gamma_j\left(2 u_z-\frac{2i\epsilon C_2f_2+ (\gamma_2)_z}{\gamma_2}\right),
\]
\[
(\bar f_2)_{z}=\bar f_2 \left(2 u_z-\frac{2i\epsilon C_2f_2+ (\gamma_2)_z}{\gamma_2}\right)+\frac{i\epsilon (-1)^{p+1} e^{2u}\bar\gamma_2C_1}{4},\]
at the points where $\gamma_2\neq 0$.
\end{Prop}
To end this subsection, we note that at the (para-)complex points with respect to the (para-)complex  structure $J_k$  of a minimal immersion $F$,  the Frenet equations and the fundamental data $(u, C_j,\gamma_j,f_j)$ are given by the Propositions \ref{p:frenequations} and \ref{p:funddata} by setting $\gamma_k=f_k=0$ and $C_k^2=1$.

\vspace{0.1in}
\begin{example}The slices ${\mathbb S}^2_p\times {q}$ and ${q}\times {\mathbb S}^2_p$ are the only surfaces that are complex with respect to both $J_1$ and $J_2.$
\end{example}\begin{example}Consider the inverse stereographic projection function from $\R^2 \to \S^2$ given by 
$$s(x,y)=\left(\frac{2 x}{x^2+y^2+1},\frac{2 y}{x^2+y^2+1},\frac{x^2+y^2-1}{x^2+y^2+1}\right).$$  Take any locally holomorphic function $w(z)=(u(x,y)+ i v(x,y))$ and consider the immersion
$F(x,y)=(s(x,y),s(u(x,y),v(x,y))).$  A straight-forward computation shows that this is a $J_1$ complex curve in $\S^2_0 \times \S^2_0$.  The only thing needed for this to work is the Cauchy Riemann equations.  The coordinates are isothermal. One has to check that the metric is  definite.  This is not automatic; if we look at $u(x,y)=-y$ and $v(x,y)=x$ the metric is totally degenerate.

 In general the immersion is definite  off a one-dimensional manifold in the domain.  For example, if $w(z)=z^2$ then 
$$G(F_x,F_x)=\frac{4 \left(3 x^4+8 x^3+6 x^2 \left(y^2+1\right)+x \left(8 y^2+4\right)+y^2 \left(3 y^2+2\right)\right)}{\left(x^2+y^2+1\right)^2 \left(2 x^2+2
   x+2 y^2+1\right)^2}$$ which is zero,  on the  circles $(x + 1)^2 + y^2 = 1$.
   \vskip .1in 
   Thus we get a wealth of distinct complex curves.
\end{example}
\begin{example} This is the same example modified for the para-complex case.  Here inverse stereographic projection is:
$$\sigma(t,s)=\left(\frac{2 t}{s^2-t^2+1},\frac{2 s}{s^2-t^2+1},\frac{s^2-t^2-1}{s^2-t^2+1}\right).$$
Take any locally para-holomorphic function $w(t+is)=(u(t,s)+ i v(t,s))$ and consider the immersion
$F(x,y)=(\sigma(x,y),\sigma(u(t,s),v(t,s))).$  The same computation shows that this is a $J_1$ para-complex curve in $d\S^2 \times d\S^2$ using the fact that $u_t=v_s$ and $u_s=v_t$.  The coordinates are isothermal. One has to check that the metric is  non-degenerate, which is not always the case.  Indeed using $u(t,s)=t/(t^2-s^2)$ and $v(t,s)=s/(t^2-s^2)$, which come from $1/z$ yields a totally degenerate metric.  But, for example $z^2$ and the function $w=s+it$ both are non-degenerate\end{example}
Again this is a large family of examples.
\subsection{The general case}

We have the following

\begin{Prop}
Let $F:\Sigma\rightarrow {\mathbb S}^2_p\times {\mathbb S}^2_p$ be a minimal immersion of an orientable surface $\Sigma$ with fundamental data $(u,C_j,\gamma_j,f_j:j=1,2)$ for a given orthonormal frame. Then
\begin{eqnarray}
|\nabla C_j|^2&=&\Big(\epsilon C_j^2+(-1)^{p+1}\Big)\Big(K+\epsilon b (-1)^{j+1}K^{\bot}+(-1)^{p+1}C_jC_{j'}\Big),\label{e:lenfthofkahler} \\
\Delta C_j&=&2\epsilon C_j\Big(K+\epsilon b (-1)^{j+1}K^{\bot}\Big)- C_{j'}\Big(1-\epsilon(-1)^{p+1}C_j^2\Big),\label{e:laplacian}
\end{eqnarray}
where $\nabla$ and $\Delta$ denote the gradient operator and the Laplacian of the induced metric $F^{\ast}G$, respectively and  $j':=3-j$.  We use this latter symbol from now on.
\end{Prop}
\begin{proof}
Assuming that the surface $\Sigma$ is minimal, the normal curvature, given in (\ref{e:norcarv}), becomes 
\[
K^{\bot}=2[G(h(e_1,e_1),v_1)G(h(e_1,e_2),v_2)-G(h(e_1,e_1),v_2)G(h(e_1,e_2),v_1)],
\]
where $h$ denotes the mean curvature of $\Sigma$, and $\{e_1,e_2\}$, $\{v_1,v_2\}$ are orthonormal frames of the tangent and the normal bundle, respectively, such that $\{e_1,e_2,v_1,v_2\}$ is an oriented orthonormal frame.


A brief computation gives,
\begin{equation}\label{e:normalcurvintermsf}
K^{\bot}=4\epsilon e^{-4u}(|f_1|^2-|f_2|^2).
\end{equation}
On the other hand, the Gauss curvature $K$ is:
\begin{equation}\label{e:curvintermsu}
K=-4\epsilon e^{-2u}u_{z\bar z}.
\end{equation}
The $F_z$-component of $F_{z\bar z z}=F_{zz\bar z}$ together with (\ref{e:normalcurvintermsf}) and (\ref{e:curvintermsu}) give
\begin{equation}\label{e:normoffk}
|f_j|^2=\frac{be^{4u}}{8}\Big(K+\epsilon b (-1)^{j+1}K^{\bot}+(-1)^{p+1}C_1C_2\Big).
\end{equation}
 Using (\ref{e:kahler}) and (\ref{e:normoffk}) we obtain (\ref{e:lenfthofkahler}).

Furthermore, differentiating (\ref{e:gammanorsec}) and using (\ref{e:derivofgamma}) we find,  
\begin{equation}\label{e:derivofgamma1}
(\gamma_j)_z=4u_z\gamma_j-\gamma_j\gamma^{-1}_{j'}\Big(2i\epsilon C_{j'}f_{j'}+(\gamma_{j'})_z\Big)-2i\epsilon C_jf_j,
\end{equation}
where $j=1,2$.

We now differentiate (\ref{e:kahler}) with respect to $\bar z$ and we use (\ref{e:deroff1}) together with (\ref{e:derivofgamma1}) to find (\ref{e:laplacian}).
\end{proof}




For totally geodesic surfaces we obtain the following:

\begin{Thm}\label{t:firsttheorem}
Let $\Sigma$ be an orientable surface in $({\mathbb S}_p^2\times {\mathbb S}_p^2,G)$. Then, the following statements are equivalent:

\noindent\emph{(a)} $\Sigma$ is $G$-totally geodesic,

\noindent\emph{(b)} $\Sigma$ is locally the product  of geodesics in ${\mathbb S}_p^2$, or it is a \emph{(}para-\emph{)}complex curve with respect to $J_1$ and $J_2$. 

\noindent\emph{(c)} The immersion $F$ is minimal and the K\"ahler functions $C_k$ are constant, both equal to $0$ or $1$.
\end{Thm}
\begin{proof} Let $F:\Sigma\rightarrow {\mathbb S}_p^2\times {\mathbb S}_p^2$ be an immersion of the orientable surface $\Sigma$ in ${\mathbb S}_p^2\times {\mathbb S}_p^2$.

$(a)\Rightarrow (b)$ Let $\Sigma$ be $G$-totally geodesic. Then, we have that $f_1=f_2=0$. From (\ref{e:kahler}), we have that $C_1$ and $C_2$ are both constant. We consider two cases:
\begin{itemize}
\item$C_1C_2\neq 0$. 
\end{itemize} 
From (\ref{e:lenfthofkahler}), we have $(\epsilon C_j^2+(-1)^{p+1})(a_j+(-1)^{p+1}C_1C_2)=0$, where $a_j=K+\epsilon b (-1)^{j+1}K^{\bot}$ and $j=1,2$. We are going to prove that $\epsilon C_j^2+(-1)^{p+1}=0$. Assuming that $\epsilon C_j^2+(-1)^{p+1}\neq 0$, we must have  $a_j=(-1)^{p}C_1C_2$ and using (\ref{e:laplacian}), we obtain
\begin{eqnarray}
0&=&\Delta C_j\nonumber \\
&=&(-1)^p C_{j'}\Big(\epsilon C_j^2+(-1)^{p+1}\Big),\nonumber
\end{eqnarray}
which gives a contradiction. Then, $\epsilon C_j^2+(-1)^{p+1}=0$ for $j=1,2$ and thus $\Sigma$ is a (para-)complex curve with respect to $J_1$ and $J_2$. 

\begin{itemize}
\item $C_1C_2=0$. 
\end{itemize} 
Assume, without loss of generality, that $C_1=0$. We prove that $C_2=0$. Since $f_1=f_2=0$, we use (\ref{e:normoffk}) to get 
\[
a_1=a_2=0,
\]
which yields $K=K^{\bot}=0$. Then from (\ref{e:laplacian}), we have that $\Delta C_1=-C_2=0$ and thus $\Sigma$ is locally the product of geodesics in ${\mathbb S}^2_p$.

$(b)\Rightarrow (c)$ This is trivial.

$(c)\Rightarrow (a)$ Suppose that the K\"ahler functions $C_1$ and $C_2$ are both constant. We prove that $F$ is totally geodesic. We use (\ref{e:kahler}) to obtain 
\begin{equation}\label{e:theproductzero}
\bar\gamma_1 f_1=\bar\gamma_2 f_2=0.
\end{equation}

\begin{itemize}
\item Assume, without loss of generality, that $\gamma_1(q)\neq 0$.
\end{itemize} 
We know that $\gamma_1\ne 0$ in a neighborhood of $a$  Then, from (\ref{e:theproductzero}), we have $f_1=0$ in that neighborhood and therefore, using (\ref{e:deroff1}) for $j=1$, we have $C_2=0$. On the other hand, from (\ref{e:laplacian}), we get 
\[
0=\Delta C_2=-C_1,
\]
and thus $C_1=C_2=0$, which is locally the product of geodesics in ${\mathbb S}^2_p$ and therefore $F$ is $G$-totally geodesic.

\begin{itemize}
\item We now assume, without loss of generality, that $\gamma_1(q)=0$. \end{itemize} 
This implies that $C_1=1$, so that 
$f_1=0$ and if, in addition, $f_2=0$, we have that $F$ is $G$-totally geodesic. If $f_2\neq 0$, then (\ref{e:theproductzero}) shows that $\gamma_2=0$ and thus $F$ is a (para-)complex curve with respect to both $J_1$ and $J_2$. Here, the structure equations of $F$ are no longer valid and so we can not apply them. We then consider local isothermic coordinates $(x,y)$ of the induced metric $F^{\ast}G$ such that $\left<F_x,F_x\right>_p=\epsilon\left<F_y,F_y\right>_p=e^{2u}$ and $\left<F_x,F_y\right>_p=0$. From the definition of the K\"ahler functions $C_1,C_2$ in (\ref{e:defiofck}), and using the fact that $J_k:T\Sigma\rightarrow T\Sigma$, we have that 
\[
J_kF_x=C_kF_y,\qquad J_kF_y=-\epsilon C_kF_x,
\]
and if $z=x+iy$, we have $J_kF_z=iC_kF_z$.

Note that $\Omega_k\wedge\Omega_k=2(-1)^k\pi_1^{\ast}\omega\wedge\pi_2^{\ast}\omega$ and the orientation of ${\mathbb S}^2_p\times {\mathbb S}^2_p$ is given by $\pi_1^{\ast}\omega\wedge\pi_2^{\ast}\omega$. Let $(N,\tilde N)$ be an orthonormal frame of the normal bundle $N\Sigma$ such that $(e^{-u}F_x,e^{-u}F_y,N,\tilde N)$ is an oriented orthonormal frame of ${\mathbb S}^2_p\times {\mathbb S}^2_p$. If $\xi=(N-i\epsilon\tilde N)/\sqrt{2}$, then by using the following relations
\[
-(\Omega_1\wedge\Omega_1)(e^{-u}F_x,e^{-u}F_y,N,\tilde N)=(\Omega_2\wedge\Omega_2)(e^{-u}F_x,e^{-u}F_y,N,\tilde N)=1,
\]
we find that $J_1\xi=-iC_1\xi$ and $J_2\xi=iC_2\xi$. Let $h$ be the second fundamental form of the immersion $F$ in ${\mathbb S}^2_p\times {\mathbb S}^2_p$. Then, since $F$ is a (para-)complex curve with respect to $J_k$, we have that $J_kh(X,Y)=h(J_kX,Y)=h(X,J_kY)$ for every tangential vector fields $X,Y$. Furthermore, 
\begin{eqnarray}
\left<h(F_z,F_z),\xi\right>&=&(-1)^p\left<J_2h(F_z,F_z),J_2\xi\right>\nonumber \\
&=&-\left<h(F_z,F_z),\xi\right>,\nonumber
\end{eqnarray}
which shows that 
\begin{equation}\label{e:firstdlk}
\left<h(F_z,F_z),\xi\right>=0.
\end{equation}
Similarly,
\begin{eqnarray}
\left<h(F_z,F_z),\bar\xi\right>&=&-\left<h(F_z,F_z),\bar\xi\right>,\nonumber
\end{eqnarray}
which also gives 
\begin{equation}\label{e:seconddlk}
\left<h(F_z,F_z),\bar\xi\right>=0.
\end{equation} 
Hence (\ref{e:firstdlk}) and (\ref{e:seconddlk}) gives $h(F_z,F_z)=0$. The fact that $F$ is $G$-minimal implies that $h(F_z,F_{\bar z})=0$ and therefore $h=0$.
\end{proof}

\vspace{0.1in}

Following \cite{torurb}, the next Proposition gives a relation between the solutions of the various Gordon equations equations and minimal surfaces.

\begin{Prop}\label{t:sinhgo}
Let $F:\Sigma\rightarrow {\mathbb S}^2_p\times {\mathbb S}^2_p$, be a minimal immersion of an orientable surface without complex points. Let $C_1$ and $C_2$ be the K\"ahler functions of $F$. 
Then, we have the following:

\vspace{0.1in}

\noindent ${\bf (1)}$ If $p\in\{0,2\}$ and the induced metric $F^{\ast}G$ is Riemannian, there exist smooth functions $v,w:\Sigma\rightarrow {\mathbb R}$ satisfying the following sinh-Gordon equations:

If $\lambda C_1C_2>0$, where $\lambda\in\{-1,1\}$, then
\[
v_{z\bar z}+\lambda\frac{|\left<J_1F_z,J_2F_z\right>|}{4}\sinh 2v=0,\quad 
w_{z\bar z}+\lambda\frac{|\left<J_1F_z,J_2F_z\right>|}{4}\sinh 2w=0.
\]

\vspace{0.1in}

\noindent ${\bf (2)}$ If $p=1$ and the induced metric $F^{\ast}G$ is Lorentzian, there exist smooth functions $v,w:\Sigma\rightarrow {\mathbb R}$ satisfying the following equations:
\begin{itemize}
\item If $|C_1|<1$ and $|C_2|<1$,
\[
v_{z\bar z}-\frac{|\left<J_1F_z,J_2F_z\right>|}{4}\sinh 2v=0,\quad
w_{z\bar z}-\frac{|\left<J_1F_z,J_2F_z\right>|}{4}\sinh 2w=0,
\]
\item If $|C_1|>1$ and $|C_2|>1$ and $\lambda C_1C_2>0$, where $\lambda\in\{-1,1\}$, then
\[
v_{z\bar z}-\lambda\frac{|\left<J_1F_z,J_2F_z\right>|}{4}\sinh 2v=0,\;\; 
w_{z\bar z}-\lambda\frac{|\left<J_1F_z,J_2F_z\right>|}{4}\sinh 2w=0.
\]


\end{itemize}

\vspace{0.1in}

\noindent ${\bf (3)}$ If $p=1$ and the induced metric $F^{\ast}G$ is Riemannian, there exist smooth functions $v,w:\Sigma\rightarrow {\mathbb R}$ satisfying:
\[
v_{z\bar z}+\frac{|\left<J_1F_z,J_2F_z\right>|}{4}\sin 2v=0,\quad
w_{z\bar z}-\frac{|\left<J_1F_z,J_2F_z\right>|}{4}\sin 2w=0.
\]

\vspace{0.1in}

\noindent ${\bf (4)}$ If $p\in\{0,2\}$ and the induced metric $F^{\ast}G$ is Lorentzian, there exist smooth functions $v,w:\Sigma\rightarrow {\mathbb R}$ satisfying:
\[
v_{z\bar z}-\frac{|\left<J_1F_z,J_2F_z\right>|}{4}\sin 2v=0,\quad
w_{z\bar z}+\frac{|\left<J_1F_z,J_2F_z\right>|}{4}\sin 2w=0.
\]

\end{Prop}
\begin{proof}
\noindent ${\bf (1)}$  Assume that $F$ is a Riemannian immersion, that is $\epsilon=1$, of a minimal surface in ${\mathbb S}^2\times {\mathbb S}^2$ or ${\mathbb H}^2\times {\mathbb H}^2$. The normal vector field $N$ satisfies $|N|^2=-1$, and thus the constant function $b$ is equal to $1$. Since $|\gamma_j|\geq 0$, from (\ref{e:gammanorsec}), we conclude that the K\"ahler functions $C_1,C_2$ satisfy:
\begin{equation}\label{e:impiequal0}
C_1^2\geq 1,\qquad\mbox{and}\qquad C_2^2\geq 1.
\end{equation}

Using (\ref{e:lenfthofkahler}), we have the following equations
\[
\Delta\log|C_j-1|=a_j+C_{j'},\qquad \Delta\log|C_j+1|=a_j-C_{j'},
\]
where
\[
a_j=K+(-1)^{j+1}K^{\bot}.
\]
Define the real functions $v,w$ by
\begin{equation}\label{e:definitionofvandw}
2v=\log\sqrt{\frac{|C_1+1||C_2+1|}{|C_1-1||C_2-1|}},\qquad 2w=\log\sqrt{\frac{|C_1-1||C_2+1|}{|C_1+1||C_2-1|}}.
\end{equation}
Therefore,
\begin{equation}\label{e:ter32}
C_1=\coth(v-w),\qquad C_2=\coth(v+w).
\end{equation}
We also have the following relations
\[
\Delta v=-\frac{\coth(v-w)+\coth(v+w)}{2},\qquad \Delta w=-\frac{\coth(v-w)-\coth(v+w)}{2}.
\]
From (\ref{e:ter32}), observe that $C_j>0$ iff $v+(-1)^jw>0$.

\vspace{0.1in}

\noindent $\bullet$  We first consider the case $C_1C_2>0$ and assume that $C_j>1$ for $j=1,2$. 
From $\left<J_1F_z,J_2F_z\right>=-\gamma_1\gamma_2$, we see
\[
|\left<J_1F_z,J_2F_z\right>|^2=\frac{e^{4u}(C_1^2-1)(C_2^2-1)}{4}.
\]
Since $C_1C_2>0$, we have that $\sinh (v-w)\sinh (v+w)>0$. In particular, 
\[
|\sinh (v-w)|=\frac{1}{\sqrt{C_1^2-1}},\qquad |\sinh (v+w)|=\frac{1}{\sqrt{C_2^2-1}}
\]
Thus,
\begin{equation}\label{e:termsofc1c22}
e^{2u}=2|\left<J_1F_z,J_2F_z\right>|\sinh (v-w)\sinh (v+w).
\end{equation}
The fact that $\Delta f=4e^{-2u}f_{z\bar z}$ together with (\ref{e:termsofc1c22}) implies that $v$ and $w$ both satisfy the equations
\[
v_{z\bar z}+\frac{|\left<J_1F_z,J_2F_z\right>|}{4}\sinh 2v=0,\qquad 
w_{z\bar z}+\frac{|\left<J_1F_z,J_2F_z\right>|}{4}\sinh 2w=0.
\]

\noindent $\bullet$  Consider now the case $C_1C_2<0$. In this case, the computations are identical with the part of $C_1C_2>0$, except that 
\begin{equation}\label{e:termsofc1c22}
e^{2u}=-2|\left<J_1F_z,J_2F_z\right>|\sinh (v-w)\sinh (v+w).
\end{equation}
Therefore, the functions $v$ and $w$ satisfy the equations
\[
v_{z\bar z}-\frac{|\left<J_1F_z,J_2F_z\right>|}{4}\sinh 2v=0,\qquad 
w_{z\bar z}-\frac{|\left<J_1F_z,J_2F_z\right>|}{4}\sinh 2w=0.
\]

\vspace{0.1in}

\noindent ${\bf (2)}$  Assume that $F$ is a Lorentzian immersion ($\epsilon=-1$) of a minimal surface in ${\mathbb S}_1^2\times {\mathbb S}_1^2$. Working away of paracomplex points, we have that $C_j^2\neq 1$. Using (\ref{e:lenfthofkahler}), we have the following relations
\[
\Delta\log|C_j-1|=-a_j+C_{j'},\qquad \Delta\log|C_j+1|=-a_j-C_{j'}.
\]
Defining $v$ and $w$ by (\ref{e:definitionofvandw}), then $C_j=\coth(v+(-1)^j w)$,  if $|C_j|>1$ and $C_j=\tanh(v+(-1)^j w)$, if $|C_j|<1$. Observe again that, $v+(-1)^j w>0$ iff $C_j>0$.

\vspace{0.1in}

\noindent $\bullet$  Let $|C_j|<1$ for $j=1,2$. Here,
\[
\cosh(v+w)=\frac{1}{\sqrt{1-C_2^2}},\quad \cosh(v-w)=\frac{1}{\sqrt{1-C_1^2}}.
\]
Then,
\[
\Delta v=-\frac{\tanh(v-w)+\tanh(v+w)}{2},\qquad \Delta w=-\frac{\tanh(v-w)-\tanh(v+w)}{2}.
\]
We also have
\[
|\left<J_1F_z,J_2F_z\right>|^2=\frac{e^{4u}(1-C_1^2)(1-C_2^2)}{4}.
\]
and thus,
\[
e^{2u}=2|\left<J_1F_z,J_2F_z\right>|\cdot\cosh (v-w)\cdot\cosh (v+w),
\]
The fact that $\Delta f=-4e^{-2u}f_{z\bar z}$ yields,
\[
v_{z\bar z}-\frac{|\left<J_1F_z,J_2F_z\right>|}{4}\sinh 2v=0,\qquad 
w_{z\bar z}-\frac{|\left<J_1F_z,J_2F_z\right>|}{4}\sinh 2w=0.
\]

\noindent $\bullet$  Let $|C_j|>1$ for $j=1,2$. If the functions $v,w$ are defined by (\ref{e:definitionofvandw}), then
\[
C_1=\coth(v-w),\qquad C_2=\coth(v+w),
\]
and
\[
\Delta v=-\frac{\coth(v-w)+\coth(v+w)}{2},\qquad \Delta w=-\frac{\coth(v-w)-\coth(v+w)}{2}.
\]
If $C_1C_2>0$ (iff $v^2-w^2>0$), then $\sinh(v-w)\sinh(v+w)>0$ and thus,
\[
e^{2u}=2|\left<J_1F_z,J_2F_z\right>|\cdot\sinh (v-w)\cdot\sinh (v+w).
\]
A brief computation gives,
\[
v_{z\bar z}-\frac{|\left<J_1F_z,J_2F_z\right>|}{4}\sinh 2v=0,\qquad 
w_{z\bar z}-\frac{|\left<J_1F_z,J_2F_z\right>|}{4}\sinh 2w=0.
\]
If now $C_1C_2<0$ (iff $v^2-w^2<0$), we have $\sinh(v-w)\sinh(v+w)<0$ and thus,
\[
e^{2u}=-2|\left<J_1F_z,J_2F_z\right>|\cdot\sinh (v-w)\cdot\sinh (v+w).
\]
Thus,
\[
v_{z\bar z}+\frac{|\left<J_1F_z,J_2F_z\right>|}{4}\sinh 2v=0,\qquad 
w_{z\bar z}+\frac{|\left<J_1F_z,J_2F_z\right>|}{4}\sinh 2w=0.
\]






\vspace{0.1in}

\noindent ${\bf (3)}$  Assume that $F$ is a Riemannian immersion ($\epsilon=1$) of a minimal surface in ${\mathbb S}_1^2\times {\mathbb S}_1^2$. Using (\ref{e:lenfthofkahler}), we have the following relations
\[
\Delta\tan^{-1} C_j=-C_{j'}.
\]
Define the following functions:
\[
v=\frac{1}{2}(\tan^{-1}C_1+\tan^{-1}C_2),\quad w=\frac{1}{2}(\tan^{-1}C_1-\tan^{-1}C_2).
\]
Then, $C_1=\tan(v+w)$ and $C_2=\tan(v-w)$ and thus,
\[
\cos^2(v+w)=\frac{1}{1+C_1^2},\quad \cos^2(v-w)=\frac{1}{1+C_2^2}.
\]
Since $v+w=\tan^{-1}C_1$ and $v-w=\tan^{-1}C_2$ we have $|v+w|<\pi/2$ and $|v-w|<\pi/2$. Thus, $\cos(v+w)>0$ and $\cos(v-w)>0$. Therefore,
\[
\cos(v+w)=\frac{1}{\sqrt{1+C_1^2}},\quad \cos(v-w)=\frac{1}{\sqrt{1+C_2^2}}.
\]
The expression,
\[
|\left<J_1F_z,J_2F_z\right>|^2=\frac{e^{4u}}{4}(C_1^2+1)(C_2^2+1),
\]
gives
\[
e^{2u}=2|\left<J_1F_z,J_2F_z\right>|\cos (v-w)\cos (v+w).
\]
In addition,
\[
\Delta v=-\frac{\tan(v-w)+\tan(v+w)}{2},\qquad \Delta w=\frac{\tan(v+w)-\tan(v-w)}{2},
\]
and using the fact that $\Delta f=4e^{-2u}f_{z\bar z}$ we finally obtain,
\[
v_{z\bar z}+\frac{|\left<J_1F_z,J_2F_z\right>|}{4}\sin 2v=0,\qquad 
w_{z\bar z}-\frac{|\left<J_1F_z,J_2F_z\right>|}{4}\sin 2w=0.
\]

\vspace{0.1in}

\noindent ${\bf (4)}$  Assume that $F$ is a Lorentzian immersion ($\epsilon=-1$) of a minimal surface in ${\mathbb S}^2\times {\mathbb S}^2$ or ${\mathbb H}^2\times {\mathbb H}^2$. We define the functions $u,v$ as for the case ${\bf (3)}$ and following the same argument and using the fact that $\Delta f=4e^{-2u}f_{z\bar z}$ we find,
\[
v_{z\bar z}+\frac{|\left<J_1F_z,J_2F_z\right>|}{4}\sin 2v=0,\qquad 
w_{z\bar z}-\frac{|\left<J_1F_z,J_2F_z\right>|}{4}\sin 2w=0.
\]\end{proof}

\vspace{0.1in}

We denote by ${\mathbb D}$ the paracomplex plane $({\mathbb R}^2,\cdot)$, where $(x_1,y_1)\cdot (x_2,y_2)=(x_1x_2+y_1y_2,x_1y_2+x_2y_1)$. We now construct a 1-parameter family of minimal surfaces that are neither (para-)complex nor Lagrangian.

\begin{Thm} We have the following statements:

\noindent \emph{(A)} Let $v,w$ be two (para-)complex real functions satisfying 
\begin{equation}\label{e:firsequa}
v_{z\bar z}+\frac{1}{2}\sinh 2v=0\quad\mbox{and}\quad w_{z\bar z}+\frac{1}{2}\sinh 2w=0,
\end{equation}
\begin{enumerate}
\item
On the region $\Gamma:=\{z\in{\mathbb C}|\; v^2(z,\bar z)-w^2(z,\bar z)>0\}$, there exists a 1-parameter family of Riemannian minimal immersions $F_t:\Gamma\rightarrow {\mathbb S}^2_p\times {\mathbb S}^2_p$, $p\in\{0,2\}$, without complex points, whose induced metric is $4\sinh(v+w)\sinh(v-w)|dz|^2$ and whose K\"ahler functions are $C_j=\coth(v+(-1)^j w)$.

\vspace{0.1in}

\item On the region $\Gamma:=\{z\in{\mathbb D}|\; v^2(z,\bar z)-w^2(z,\bar z)<0\}$, there exists a 1-parameter family of Lorentzian minimal immersions $F_t:\Gamma\rightarrow {\mathbb S}^2_1\times {\mathbb S}^2_1$, without paracomplex points, whose induced metric is $-4\sinh(v+w)\sinh(v-w)|dz|^2$ and whose K\"ahler functions are $C_j=\coth(v+(-1)^j w)$.

\end{enumerate}

\noindent \emph{(B)} Let $v,w$ be two (para-)complex real functions satisfying 
\begin{equation}\label{e:firsequa1}
v_{z\bar z}-\frac{1}{2}\sinh 2v=0\quad\mbox{and}\quad w_{z\bar z}-\frac{1}{2}\sinh 2w=0,
\end{equation}
\begin{enumerate}
\item
On the region $\Gamma:=\{z\in{\mathbb D}|\; |v-w|<1,\; |v+w|<1\}$, there exists a 1-parameter family of Lorentzian minimal immersions $F_t:\Gamma\rightarrow {\mathbb S}^2_1\times {\mathbb S}^2_1$, without paracomplex points, whose induced metric is $4\cosh(v+w)\cosh(v-w)|dz|^2$ and whose K\"ahler functions are $C_j=\tanh(v+(-1)^j w)$.

\vspace{0.1in}

\item On the region $\Gamma:=\{z\in{\mathbb D}|\; v^2(z,\bar z)-w^2(z,\bar z)>0,\; |v-w|>1,\; |v+w|>1\}$, there exists a 1-parameter family of Lorentzian minimal immersions $F_t:\Gamma\rightarrow {\mathbb S}^2_1\times {\mathbb S}^2_1$, without paracomplex points, whose induced metric is $4\sinh(v+w)\sinh(v-w)|dz|^2$ and whose K\"ahler functions are $C_j=\coth(v+(-1)^j w)$.

\end{enumerate}

\noindent \emph{(C)} Let $v,w$ be two (para-)complex real functions satisfying 
\begin{equation}\label{e:firsequa2}
v_{z\bar z}+\frac{1}{2}\sin 2v=0\quad\mbox{and}\quad w_{z\bar z}-\frac{1}{2}\sin 2w=0,
\end{equation}
\begin{enumerate}
\item
On the region $\Gamma:=\{z\in{\mathbb C}|\; |v-w|<\pi/2,\; |v+w|<\pi/2\}$, there exists a 1-parameter family of Riemannian minimal immersions $F_t:\Gamma\rightarrow {\mathbb S}^2_1\times {\mathbb S}^2_1$, without complex points, whose induced metric is $4\cos(v+w)\cos(v-w)|dz|^2$ and whose K\"ahler functions are $C_j=\tan(v+(-1)^{j+1} w)$.

\vspace{0.1in}

\item On the region $\Gamma:=\{z\in{\mathbb D}|\; |v-w|<\pi/2,\; |v+w|<\pi/2\}$, there exists a 1-parameter family of Lorentzian minimal immersions $F_t:\Gamma\rightarrow {\mathbb S}^2_p\times {\mathbb S}^2_p$, where $p\in\{0,2\}$, without paracomplex points, whose induced metric is $4\cos(v+w)\cos(v-w)|dz|^2$ and whose K\"ahler functions are $C_j=\tan(v+(-1)^{j+1} w)$.

\end{enumerate}
\end{Thm}
\begin{proof}
(A.1.) Let $v,w$ be solutions of (\ref{e:firsequa}). For $j=1,2$, we define the following functions on $\Gamma=\{z\in{\mathbb C}|\; v^2(z,\bar z)-w^2(z,\bar z)>0\}$:
\[
C_j=\coth(v+(-1)^jw),\qquad \gamma_j=\sqrt{2}e^{it/2}\sqrt{\frac{\sinh(v+(-1)^{j+1}w)}{\sinh(v+(-1)^jw)}},
\]
\[
e^{2u}=4\sinh(v-w)\sinh(v+w),\qquad f_j=-i\gamma_j\frac{\partial}{\partial z} (v+(-1)^jw).
\]

Then, $C_j^2>1$ and $2|\gamma_j|^2=e^{2u}(C_j^2-1)$. Furthermore, for each $t\in {\mathbb R}$ the octuple $(u,A,C_j,\gamma_j,f_j)$ satisfy the compatibility equations (\ref{e:kahler}),  (\ref{e:deroff1}) and  (\ref{e:derivofgamma}) for $\epsilon=b=1$ and $p\in\{0,2\}$ for $A$ given by
\[
A=\frac{\partial}{\partial z}\left(\log\sqrt{\frac{\sinh(v+w)}{\sinh(v-w)}}\right).
\]
Note that the complex function also satisfies (\ref{e:themissing}).
Thus, from Proposition \ref{p:compatibility}, the septuple $(u,C_j,\gamma_j,f_j)$ is the fundamental data of a unique Riemannian minimal immersion $F_t:\Gamma\rightarrow {\mathbb S}_p^2\times {\mathbb S}_p^2$ with $p\in\{0,2\}$. Moreover, for each $t\in {\mathbb R}$, the immersion $F_t$ is non-Lagrangian without complex points.

The cases (A.2),  (B.1) , (B.2) and (B.3) are similar to (A.1).

\vspace{0.1in}

\noindent (C.1.) Let $v,w$ be solutions of (\ref{e:firsequa2}). For $j=1,2$, we define the following functions on $\Gamma=\{z\in{\mathbb C}|\; |v-w|<\pi/2,\; |v+w|<\pi/2\}$:
\[
C_j=\tan(v+(-1)^{j+1}w),\qquad \gamma_j=\sqrt{2}e^{it/2}\sqrt{\frac{\cos(v+(-1)^{j}w)}{\cos(v+(-1)^{j+1}w)}},
\]
\[
e^{2u}=4\cos(v-w)\cos(v+w),\qquad f_j=i\gamma_j\frac{\partial}{\partial z} (v+(-1)^{j+1}w).
\]
In this case, $2|\gamma_j|^2=e^{2u}(C_j^2+1)$. For each $t\in {\mathbb R}$, the septuple $(u,C_j,\gamma_j,f_j)$ satisfy the compatibility equations (\ref{e:kahler}),  (\ref{e:deroff1}) and  (\ref{e:derivofgamma}) for $\epsilon=b=1$ and $p=1$ and 
\[
A=\frac{\partial}{\partial z}\left(\log\sqrt{\frac{\cos(v-w)}{\cos(v+w)}}\right),
\] 
which a brief computation shows that satisfies (\ref{e:themissing}). Thus, from Proposition \ref{p:compatibility}, the septuple $(u,C_j,\gamma_j,f_j)$ is the fundamental data of a unique Riemannian minimal immersion $F_t:\Gamma\rightarrow {\mathbb S}_1^2\times {\mathbb S}_1^2$. Furthermore, for each $t\in {\mathbb R}$, the immersion $F_t$ is non-Lagrangian without complex points. 

Following a similar argument with the case (C.1), we prove (C.2).
\end{proof}

\vspace{0.1in}

\subsection{Compact surfaces}

By a compact surface we mean that the surface is closed, oriented and connected.
For a $G$-minimal immersion of a compact surface in ${\mathbb S}^2\times {\mathbb S}^2$ or ${\mathbb H}^2\times {\mathbb H}^2$, we have the following theorem:
\begin{Thm}\label{t:secondthe}
Let $\Sigma$ be a compact two-manifold, $G$-minimally immersed in ${\mathbb S}_p^2\times {\mathbb S}_p^2$, with  $p\in\{0,2\}$.

\noindent $(1)$ If the induced metric is Riemannian, then $\Sigma$ is a complex curve. In particular, $\Sigma$ is a topological type of a sphere.

\noindent $(2)$ If the induced metric is Lorentzian then for every $j=1,2$, there exists a $\Omega_j$-Lagrangian curve.
\end{Thm}
\begin{proof}
\noindent (1) Let $F:\Sigma\rightarrow {\mathbb S}_p^2\times {\mathbb S}_p^2$ be a $G$-minimal immersion such that $F^{\ast}G$ is Riemannian. As we have seen in the proof of  Proposition \ref{t:sinhgo}, when $p\in\{0,2\}$ the K\"ahler functions $C_1,C_2$ satisfy (\ref{e:impiequal0}). The expressions (\ref{e:lenfthofkahler}) and (\ref{e:laplacian}) become
\begin{equation}\label{e:impiequal}
|\nabla C_j|^2=(C_j^2-1)\Big(K+(-1)^{j+1}K^{\bot}-C_1C_2\Big),
\end{equation}
\begin{equation}\label{e:impiequal1}
\Delta C_j=2C_j\Big(K+(-1)^{j+1}K^{\bot}\Big)-C_{j'}(1+C_j^2).
\end{equation}
We first prove the existence of a $J_j$ complex point, for every $j=1,2$. Assume that $\Sigma$ has no $J_j$ complex points, for some $j=1,2$. Then $C_j^2(x)>1$, for every $x\in \Sigma$. Let $M$ be a maximum point of $C_j$ and assume that $M$ is not $J_j$-complex. Since $M$ is a critical point of $C_j$, from (\ref{e:impiequal}) we have that 
\begin{equation}\label{e:impiequal2}
K(M)+(-1)^{j+1}K^{\bot}(M)=C_1(M)C_2(M).
\end{equation}
On the other hand, we have that $\Delta C_j(M)\leq 0$ and, using (\ref{e:impiequal1}) and (\ref{e:impiequal2}), we obtain the following inequality:
\[
C_{j'}(M)(C_j^2(M)-1)\leq 0.
\]
By the assumption that $M$ is a non-$J_j$-complex point, we have that $C_j^2(M)>1$ and thus, $C_{j'}(M)\leq 0$. The fact that $\Sigma$ is connected implies that $C_{j'}(x)\leq -1$ for all $x\in\Sigma$.

Let $m$ be the minimum point of $C_j$. Then, (\ref{e:impiequal}) gives 
\begin{equation}\label{e:impiequal22}
K(m)+(-1)^{j+1}K^{\bot}(m)=C_1(m)C_2(m).
\end{equation}
Using the fact that $\Delta C_j(m)\geq 0$, the equations (\ref{e:impiequal1}) and (\ref{e:impiequal22}) give
\[
C_{j'}(m)(C_j^2(m)-1)\geq 0,
\]
which implies that $C_{j'}(m)\geq 0$ and therefore we have a contradiction.

We now prove that $\Sigma$ is a topological type of a sphere.

If $F$ is a complex curve with respect to both complex structures then $\Sigma$ must have the same genus as ${\mathbb S}_p^2$. Then for $p=0$, the compact surface $\Sigma$ is a topological type of a sphere. Note that for the case $p=2$ no such immersion exists,  since the hyperbolic plane ${\mathbb H}^2$ is non-compact.

Assume now that $F$ is a non-complex curve with respect to $J_j$ (then it must have isolated $J_j$-complex points) and suppose that $p\in\Sigma$ is a critical point of $C_j$. If $y$ is a complex point, then $C_j^2(y)=1$, and thus $y$ is a local maximum or minimum of $C_j$ (because $C_j^2(x)\geq 1$, for every $x\in\Sigma$ and $\Sigma$ is connected). Assuming that $y$ is a critical, non-$J_j$-complex point, from (\ref{e:impiequal}) we get
\begin{equation}\label{e:niceequality}
K(y)+(-1)^{j+1}K^{\bot}(y)=C_1(y)C_2(y).
\end{equation}
Substituting (\ref{e:niceequality}) into (\ref{e:impiequal1}), we have
\[
\Delta C_j(y)=2C_{j'}(y)(C_j^2(y)-1),
\]
which is non-zero and thus, $y$ is a local maximum or minimum. Then every critical point is either local minimum or local maximum and using  Morse's Theorem, we conclude that $\Sigma$ is a topological type of a sphere. 

We now prove that $F$ must be a complex curve. Consider the Hopf differential $\Theta$ associated to the immersion $F$ given by,
\[
\Theta(z)=\frac{1}{2}\left<J_1F_z,J_2F_z\right>dz\otimes dz,
\]
where $z$ is the conformal parameter. From (\ref{e:exprgamma1}), we see that $\left<J_1F_z,J_2F_z\right>=-\gamma_1\gamma_2$ and using (\ref{e:derivofgamma}), we see that $\Theta$ is holomorphic. Because $\Sigma$ is a sphere, it follows from Riemann-Roch's Theorem that $\Theta$ is must be indentically zero throughout $\Sigma$ and thus $\gamma_1\gamma_2=0$ which means that $F$ must be a complex curve.

\vspace{0.1in}

\noindent (2) In this case, the expressions (\ref{e:lenfthofkahler}) and (\ref{e:laplacian}) become
\begin{equation}\label{e:impiequal2}
|\nabla C_j|^2=-(C_j^2+1)\Big(K+b(-1)^{j}K^{\bot}-C_1C_2\Big),
\end{equation}
\begin{equation}\label{e:impiequal3}
\Delta C_j=-2C_j\Big(K+b(-1)^{j}K^{\bot}\Big)-C_{j'}(1-C_j^2).
\end{equation}
Then,
\begin{eqnarray}
\Delta\tan^{-1}C_j&=&\frac{(C_j^2+1)\Delta C_j-2C_j|\nabla C_j|^2}{(C_j^2+1)^2},\nonumber
\end{eqnarray}
and using (\ref{e:impiequal2}) with (\ref{e:impiequal3}), we obtain
\[
\Delta\tan^{-1}C_j=-C_{j'}.
\] 
Hence,
\[
\int_{\Sigma}C_1=\int_{\Sigma}C_2=0,
\]
and thus $C_1$ and $C_2$ must vanishes at least in curves on $\Sigma$.
\end{proof}

\begin{Cor}
Every compact $G$-minimal surface in ${\mathbb S}^2\times {\mathbb S}^2$ or in ${\mathbb H}^2\times {\mathbb H}^2$ is either a topological type of a sphere or a torus.
\end{Cor}

For the case of $d{\mathbb S}^2\times d{\mathbb S}^2$, we obtain the following:

\begin{Thm}\label{t:secondthe2}
Every compact $G$-minimal surface $\Sigma$ in $d{\mathbb S}^2\times d{\mathbb S}^2$ is a topological type of a torus. If furthermore, the metric $G$ induced on $\Sigma$ is Riemannian then for every $j=1,2$, there exists a $\Omega_j$-Lagrangian point. 
\end{Thm}
\begin{proof}
Clearly $\Sigma$ is a tori when the induced metric is Lorentzian. Assume that the induced metric is Riemannian. Let $m$ and $M$ be the minimum and the maximum point of $C_j$, respectively. Then, substituting $\epsilon=p=1$ into (\ref{e:lenfthofkahler}) and (\ref{e:laplacian}) and using the fact that $m$ is a minimum point of $C_j$ we get $a_j(m)=-C_1(m)C_2(m)$ and $C_{j'}(m)\geq 0$. Following a similar argument for the point $M$, we obtain that $C_{j'}(M)\leq 0$. Since $\Sigma$ is connected, there exists $x\in\Sigma$ such that $C_{j'}(x)=0$. We now prove that $\Sigma$ is a topological type of a torus. We again consider the Hopf differential $\Theta$ associated to the immersion $F$ which is holomorphic. On the other hand, $\Theta$ is non-vanishing on $\Sigma$ since,
\[
|\left<J_1F_z,J_2F_z\right>|^2=|\gamma_1|^2|\gamma_2|^2=e^{4u}(C^2_1+1)(C^2_2+1)/4,
\]
and thus the genus is one.
\end{proof}

Finally, the following Theorem gives a rigidity result for the function $K\pm K^{\bot}$ for compact minimal Riemannian surfaces in $d{\mathbb S}^2\times d{\mathbb S}^2$.

\begin{Thm}
Let $F:\Sigma\rightarrow d{\mathbb S}^2\times d{\mathbb S}^2$ be a $G$-minimal immersion of a compact surface $\Sigma$ such that the induced metric is Riemmanian. For $m\in\{1,2\}$ assume that we have $K+(-1)^{m}K^{\bot}\geq 0$ or $K+(-1)^{m}K^{\bot}\leq 0$. Then, $K=K^{\bot}=0$ and $F$ is locally the product of geodesics in $d{\mathbb S}^2$. 
\end{Thm}
\begin{proof}
Consider a $G$-minimal immersion $F:\Sigma\rightarrow d{\mathbb S}^2\times d{\mathbb S}^2$ such that the surface $\Sigma$ is compact and the induced metric  $F^{\ast}G$ is Riemannian. In this case, we have
\begin{equation}\label{e:simkl1}
|\nabla C_{m'}|^2=(C_{m'}^2+1)(K+(-1)^{m}K^{\bot}+C_1C_2),
\end{equation}
and,
\begin{equation}\label{e:simkl2}
\Delta C_{m'}=2C_{m'}(K+(-1)^{m}K^{\bot})-C_{m}(1-C_{m'}^2).
\end{equation}
A brief computation gives
\[
\Delta\log\sqrt{1+C_{m'}^2}=K+(-1)^{m}K^{\bot},
\]
and thus,
\[
\int_{\Sigma}K+(-1)^{m}K^{\bot}=0.
\]
Assuming, without loss of generality that $K+(-1)^{m}K^{\bot}\geq 0$, we have
\begin{equation}\label{e:esteexpre}
K+(-1)^{m}K^{\bot}= 0.
\end{equation}
Using (\ref{e:esteexpre}), the expressions (\ref{e:simkl1}) and (\ref{e:simkl2}) become
\begin{equation}\label{e:esteexpre1}
|\nabla C_{m'}|^2=(C_{m'}^2+1)C_1C_2,\quad\mbox{and}\quad \Delta C_{m'}=C_m(C_{m'}^2-1),
\end{equation}
and since $|\nabla C_{m'}|^2\geq 0$ we have $C_1C_2\geq 0$.


The relations (\ref{e:esteexpre1}) yield
\[
\Delta C_{m'}^2=4C_1C_2 C_{m'}^2\geq 0.
\]
Thus, $C_{m'}^2\Delta C_{m'}^2\geq 0$ and therefore, the K\"ahler function $C_{m'}$ is constant. Then,
\[
\bar\gamma_{m'} f_{m'}=0,
\]
and since 
\[
|\gamma_{m'}|^2=\frac{e^{2u}(C_{m'}^2+1)}{2},
\]
we have that $f_{m'}=0$. Hence, $\bar\gamma_{m'} C_{m}=0$, and thus $C_{m}=0$. Then $F$ is Lagrangian and therefore it is the product of geodesics in $d{\mathbb S}^2$. 
\end{proof}


\end{document}